\documentclass{elsarticle}
\usepackage[letter, center, frame]{crop}
\usepackage{amssymb, amsmath, amsthm, latexsym}
\usepackage{setspace}
\usepackage{eucal}
\usepackage{fullpage}
\usepackage{cancel}
\usepackage{ytableau}

\newtheorem{thm}{Theorem}[section]
 \newtheorem{cor}[thm]{Corollary}
 \newtheorem{lem}[thm]{Lemma}
 \newtheorem{prop}[thm]{Proposition}
 
 \theoremstyle{definition}
 
 \theoremstyle{remark}

 \numberwithin{equation}{section}

 \title{Relating the type A alcove path model to the right key of a semistandard Young tableau, with Demazure character consequences}

\author[MJW]{Matthew J. Willis}
\ead[MJW]{mwillis1@conncoll.edu}

\address[MJW]{Dept. of Mathematics, Connecticut College, New London, CT 06320, USA}

\begin{document}

\begin{abstract}
There are several combinatorial methods that can be used to produce type A Demazure characters (key polynomials).  The alcove path model of Lenart and Postnikov provides a procedure that inputs a semistandard tableau $T$ and outputs a saturated chain in the Bruhat order.  The final permutation in this chain determines a family of Demazure characters for which $T$ contributes its weight.  Separately, the right key of $T$ introduced by Lascoux and Sch{\"u}tzenberger also determines a family of Demazure characters for which $T$ contributes its weight.  In this paper we show that the final permutation in the chain produced by the alcove model corresponds bijectively to the right key of the tableau.  From this it follows that the generating sets for the Demazure characters produced by these two methods are equivalent.
\end{abstract}


\maketitle

\section{Introduction}

In their 1990 paper \cite{LS} Lascoux and Sch\"{u}tzenberger introduced the notion of the ``right key'' of a semistandard Young tableau.  One of the foremost applications of the right key is presented as Theorem 1 of \cite{RS1}, which provides a type A Demazure character formula that sums over a set of semistandard Young tableaux whose right keys satisfy a certain condition.  Several equivalent methods have since been introduced to compute the right key of a semistandard tableau $T$.  The method from \cite{Wil} produces the ``scanning tableau'' $S(T)$ for $T$, which was shown to equal the right key of $T$.

In their 2007 paper \cite{LP} Lenart and Postnikov introduced the alcove path model.  Among many other applications, this model can be used to produce Demazure characters in arbitrary type.  When specialized to type A, the Demazure character is given as a sum over certain ``admissible subsets''.  The type A ``filling map'' is described in Section 3 of \cite{Le2}.  Its inverse inputs a semistandard tableau and outputs a saturated chain in the Bruhat order.  These chains are in bijection with the admissible subsets.

The main result of this paper is as follows:  Given a semistandard tableau $T$, find its scanning tableau $S(T)$.  The scanning tableaux are in bijection with certain permutations; denote the permutation for $S(T)$ by $\sigma_T$.  Then for the same $T$ apply the inverse of the filling map to produce its saturated Bruhat chain, denoted $B_T$.  Theorem \ref{finalequivalence} states that the final permutation in the chain $B_T$ is $\sigma_T$.  Thus we have proved that the final permutation in $B_T$, which plays a role in the alcove model world analogous to the role of the right key in the tableau world, has a key tableau that is indeed equal to the right key of Lascoux and Sch\"{u}tzenberger.  The conjecture of this equality arose during discussions with Lenart.  The connection between the two subjects is obtained here by forming the inverse of the filling map from \cite{Le2}.  The results presented in Section 5 make this connection completely explicit.  From the main result it will follow that not only are the Demazure characters produced by these two methods equal, but their generating sets are as well.  We achieve this by providing a set of semistandard tableaux that is in direct correspondence with the appropriate admissible subsets.  This set of tableaux is seen to equal the set of tableaux from Theorem 1 of \cite{RS1} mentioned above.

This paper is organized as follows:  Section 2 provides the necessary familiar definitions for our work.  Section 3 recalls the ``scanning method'' of \cite{Wil} and introduces some new terminology for it.  Section 4 provides a slightly simplified version of the inverse of the filling map from the type A specialization of the alcove model.  Section 5 proves the main result, which is the relationship between the methods in Sections 3 and 4.  Section 6 gives the details of the Demazure character equalities that are a consequence of the main result.

\section{Background definitions}

Fix a positive integer $n$, and consider it fixed henceforth.  An \emph{n-partition} $\lambda = (\lambda_1, ... , \lambda_n)$ is a sequence of weakly decreasing non-negative integers.  Let $\Lambda_n^+$ denote the set of all $n$-partitions with $\lambda_n = 0$.  (Arbitrary $n$-partitions are fine for Sections 2 through 5, but we will restrict to these partitions that correspond to the dominant weights of type $A_{n-1}$.)  Fix a non-zero $\lambda \in \Lambda_n^+$.  The \emph{Young diagram of $\lambda$} is a diagram consisting of $\lambda_i$ left-justified empty boxes in the $i^{th}$ row for $1 \leq i \leq n-1$.  Henceforth we will simply use $\lambda$ to refer to its Young diagram.  Define $c_i$ to be the number of boxes in the $i^{th}$ column of $\lambda$ for $1 \leq i \leq \lambda_1$, i.e. the \emph{length} of the $i^{th}$ column of $\lambda$.  Let $1 \leq \zeta_1 < ... < \zeta_d \leq n-1$ denote the \emph{distinct column lengths of} $\lambda$.  Set $\zeta_0 := 0$ and $\zeta_{d+1} := n$. Let $\beta_h$ denote the index of the rightmost column of length $\zeta_h$ for $1 \leq h \leq d$, and set $\beta_{d+1} := 1$.

Let $(j,i)$ denote the intersection of the $j^{th}$ column and $i^{th}$ row of $\lambda$.  (We reverse from the normal convention because the columns play a larger role than the rows in this paper.)  Write $(j,i) \in \lambda$ if and only if $1 \leq j \leq \lambda_1$ and $i \leq c_j$.  Define a reading order on $\lambda$ by $(l,k) \leq (j,i)$ if $l < j$ or $l = j$ and $k \geq i$.  In this case we say that the location $(l,k)$ occurs (weakly) before $(j,i)$, and thus $(j,i)$ occurs (weakly) after $(l,k)$.  We will refer to \emph{advancing a location}, by which we mean increasing the location via this ordering to one that occurs after it.  The location immediately following $(j,i)$ is $(j,i-1)$, and the location immediately preceding it is $(j,i+1)$.  For the sake of convention, identify $(j,i-1)$ with $(j+1, c_{j+1})$ when $i = 1$, and identify $(j,i+1)$ with $(j-1,1)$ when $i = c_j$.

A \emph{filling} of $\lambda$ is an assignment of one number to each box in $\lambda$.  Define the set $[n] := \{ 1, 2, ... , n \}$.  An \emph{$n$-semistandard tableau} $T$ is a filling of $\lambda$ with values from $[n]$ such that the values weakly increase from left to right within each row and strictly increase top to bottom within each column.  In this case $\lambda$ is called the \emph{shape} of $T$.  Let $\mathcal{T}_\lambda$ denote the set of all $n$-semistandard tableaux with shape $\lambda$.  Given $T \in \mathcal{T}_\lambda$, let $T(j,i)$ be the value in $T$ at the location $(j,i) \in \lambda$.  Use $C_1, ... , C_{\lambda_1}$ to denote the columns of $T$ from left to right; hence the length of $C_i$ is $c_i$.  We say that $T$ is a \emph{key} if the values in $C_i$ also appear in $C_{i-1}$ for $1 < i \leq \lambda_1$.  The \emph{right key of $T$}, denoted $R(T)$, is a key determined by the values of $T$ that was introduced by Lascoux and Sch\"{u}tzenberger in \cite{LS}.  (We will not need a computational definition for $R(T)$.)

A \emph{permutation} is a bijection from $[n]$ to itself, and the set of all permutations is denoted $S_n$.  Given $\phi \in S_n$ we will often refer to its \emph{one-rowed form} $(\phi_1, ... , \phi_n)$; here $\phi_i$ is the image of $i$ under $\phi$ for $1 \leq i \leq n$.  Define $S_n^\lambda$ to be the set of all $\phi \in S_n$ such that $\phi_{\zeta_{h-1}+1} < ... < \phi_{\zeta_h}$ for $1 \leq h \leq d+1$.  (This is the set of minimal coset representatives of $S_n \slash S_\lambda$, where $S_\lambda$ is the subgroup of permutations that fix $\lambda$.)  Define the \emph{$\lambda$-key of $\phi$}, denoted $Y_\lambda(\phi)$, to be the key of shape $\lambda$ whose columns of length $\zeta_h$ contain $\phi_1, ... , \phi_{\zeta_h}$ arranged in increasing order for $1 \leq h \leq d$.  For an example, refer to Figure 1 in Section 4.  The figure contains the one-rowed forms of several permutations, written vertically.  Let $\phi$ be the rightmost permutation and let $\lambda = (4,4,3,2,1,1,1)$.  Then the rightmost tableau in Figure 1 is $Y_\lambda(\phi)$.

If $\lambda$ is not strict, i.e. not all parts are distinct, then there are multiple permutations in $S_n$ that will yield the same key.  Specifically, all permutations in a given coset of $S_n \slash S_\lambda$ have the same key.  By the construction of $S_n^\lambda$ and the definition of the key of a permutation, we see that if multiple permutations produce the same key, the shortest permutation in the Bruhat order that produces this key is in $S_n^\lambda$.  It is well known that this key construction is a bijection from $S_n^\lambda$ to the set of keys of shape $\lambda$.  Further, it can be seen that for $\phi, \psi \in S_n^\lambda$, one has $Y_\lambda(\phi) \leq Y_\lambda(\psi)$ if and only if $\phi \leq \psi$ in the Bruhat order \cite{BB}.

\section{The Scanning Method}

The following method inputs a semistandard tableau $T$ of shape $\lambda$ and outputs its ``scanning tableau'' $S(T)$, which is a key of the same shape.  By Theorem 4.5 of \cite{Wil}, the scanning tableau is equal to the right key $R(T)$ of $T$.

Fix $\lambda \in \Lambda_n^+$ and let $T \in \mathcal{T}_\lambda$.  Fix $1 \leq j \leq \lambda_1$, as the procedure is applied once to each column of $\lambda$.  Initialize the \emph{scanning paths} from the $j^{th}$ column by $P(T;j,i) := \{ (j,i) \}$ for $1 \leq i \leq c_j$.  Technically speaking the scanning paths are sets of locations in $\lambda$, but we will also refer to the values in a scanning path, which are simply the values in $T$ at the locations in the path.  The paths are constructed from bottom to top using the following definition:  Given a sequence $x_1, x_2, ...$, define its \emph{earliest weakly increasing subsequnce (EWIS)} to be the sequence $x_{a_1}, x_{a_2}, ...$, where $a_1 = 1$, and for $b > 1$ the index $a_b$ is the smallest index such that $x_{a_b} \geq x_{a_{b-1}}$.  Consider the values $T(l, c_l)$ for $l \geq j$ to form a sequence and compute its EWIS.  Each time a value is added to this EWIS, append its location to $P(T;j,c_j)$.  When this process terminates, delete the values and boxes in $P(T;j,c_j)$ from $T$ and $\lambda$, then repeat the process for the new lowest box in $C_j$.

In general, to compute $P(T;j,i)$ for $1 \leq i < c_j$:  Compute and then delete $P(T;j,k)$ for $c_j \geq k > i$.  Use the values in the lowest box of the $j^{th}$ through rightmost columns of the resulting tableau to create a sequence, and compute its EWIS.  Each time a value is added to the EWIS, append its location to $P(T;j,i)$.

For an example, let $T$ be the first tableau in Figure 1, located in Section 4.  Its scanning paths that begin in its first column are indicated in the second tableau in Figure 1:  A superscript of $x$ on $T(j,i)$ indicates that $(j,i) \in P(T;1,x)$.  To compute $P(T;1,i)$ using this figure, one must imagine that the entries with superscripts larger than $i$ and their boxes have been deleted from $T$ and $\lambda$.

It can be seen that the result of deleting a scanning path always leaves a valid shape and a semistandard tableau.  Further, once $P(T;j,1)$ has been computed and deleted then $C_j$ through $C_{\lambda_1}$ have been deleted.  In other words, every location $(l,k) \geq (j,c_j)$ is contained in exactly one scanning path that begins in the $j^{th}$ column. To create the \emph{scanning tableau $S(T)$}, apply the above process once for each column of $T$.  Then define $S(T;j,i)$ (the value in $S(T)$ at the location $(j,i)$ ) to be $T(l,k)$, where $(l,k)$ is the final location in $P(T;j,i)$.  Continuing the example, the third tableau in Figure 1 is $S(T)$.

The scanning tableau can be used to produce a permutation via the inverse of the bijection described in Section 2.  Fix $T \in \mathcal{T}_\lambda$ and find its scanning tableau $S(T)$.  Define $\sigma_T$ to be the permutation such that the values $\sigma_{\zeta_{h-1}+1}, ... , \sigma_{\zeta_h}$ for $1 \leq h \leq d$ are the values in the columns of length $\zeta_h$ of $S(T)$ that are not in the columns of length $\zeta_{h-1}$, arranged in increasing order, and $\sigma_{\zeta_d + 1} , ... , \sigma_{\zeta_{d+1}} = \sigma_n$ are the values from $[n]$ that do not appear in the first column of $S(T)$, arranged in increasing order.  Since $S(T)$ is a key, this process is well-defined.  By construction we have $\sigma_T \in S_n^\lambda$.  Further,  we see that $S(T) = Y_\lambda(\sigma_T)$.  Continuing the example, $\sigma_T$ is the rightmost permutation in Figure 1.

Given a column $l < \lambda_1$ and a location $(j,i) > (l, 1)$, for $1 \leq k \leq c_l$ define the \emph{most recent location of $P(T;l,k)$ relative to $(j,i)$} to be the latest location in $P(T;l,k)$ that occurs before $(j,i)$.  Computationally, this restriction simply truncates the sequences from which the EWIS's will be computed, and hence truncates the $P(T;l,k)$.  The \emph{most recent value of $P(T;l,k)$ relative to $(j,i)$} is the value in $T$ at the most recent location of $P(T;l,k)$ relative to $(j,i)$.  Continuing the example, in the second tableau let $(j,i) = (3,3)$, $l = 1$, and $k = 5$. Then the most recent location of $P(T;1,5)$ relative to $(3,3)$ is $(2,3)$, and the most recent value is 7.

\begin{lem}\label{recent}The most recent value of $P(T;l,k)$ relative to $(j,i)$ decreases as $k$ decreases.\end{lem}

\begin{proof}Fix $1 \leq l < \lambda_1$.  Let $1 \leq h < k \leq c_l$ and $(j,i) > (l,1)$.  Let $(a,b)$ and $(x,y)$ be the most recent locations of $P(T;l,h)$ and $P(T;l,k)$ relative to $(j,i)$ respectively.  Since $P(T;l,k)$ is computed before $P(T;l,h)$, the location $(a,b)$ is in the shape used to compute $P(T;l,k)$ and was not appended to $P(T;l,k)$.  Let $(a, b^\prime)$ denote the bottom location of the $a^{th}$ column when $P(T;l,k)$ was computed, i.e. the value $T(a,b^\prime)$ is in the sequence whose EWIS is used to determine $P(T;l,k)$.  It is easy to see from its definition that the final value in the EWIS obtained from a finite sequence is the largest value in that sequence.  Further, the EWIS contains all occurrences of the largest value.  From this we have $T(x,y) \geq T(a, b^\prime)$.

If $b^\prime = b$, then $T(x,y) = T(a,b)$ would mean that $(a,b)$ was appended to $P(T;l,k)$, which is false.  Thus $T(x, y) > T(a,b)$.  Assume $b^\prime > b$.  By the column-strict condition on $T$, we have $T(a,b^\prime) > T(a, b)$.  Thus $T(x,y) \geq T(a, b^\prime) > T(a,b)$.  \end{proof}

\begin{lem}\label{recentpath}Fix $1 \leq l < \lambda_1$ and $(j,i) > (l, 1)$.  The path from the $l^{th}$ column that contains $(j,i)$ is the path with the largest most recent value relative to $(j,i)$ that is less than or equal to $T(j,i)$.\end{lem}

\begin{proof}As $k$ decrements from $c_l$, compute and delete the $P(T;l,k)$ until $(j, i+1)$ has been deleted.  The column-strict condition on $T$ and the definition of EWIS imply that the most recent values relative to $(j,i)$ from these paths are larger than $T(j,i)$.  The definition of EWIS also guarantees that the remaining paths will not append $(j,i)$ if their most recent value relative to $(j,i)$ is larger than $T(j,i)$.  Continue to compute and delete the scanning paths until reaching the first path whose most recent value relative to $(j,i)$ is less than or equal to $T(j,i)$.  At least one such path must exist, namely the path that contains $(j-1,i)$.  Lemma \ref{recent} implies that this most recent value relative to $(j,i)$ is the largest such value.  Clearly this EWIS will choose $T(j,i)$ and so this path will append $(j,i)$.  \end{proof}

\section{Applying the alcove model in type A to semistandard Young tableaux}

The ``filling map'' from the type A specialization of the alcove model is presented in Section 3 of \cite{Le2}.  Its inverse inputs a tableau $T$ of shape $\lambda$ and outputs a saturated chain in the Bruhat order.  We are most interested in the final permutation in the chain, which will be denoted $\pi_T$.  This inverse procedure consists of repeated application of the ``greedy algorithm'' described in Algorithm 4.9 of \cite{Le2} to the locations in $T$ in a certain order.  The author is indebted to Lenart for explaining how to use Algorithm 4.9 to describe the inverse of the filling map.  Here we provide the details of a shortened version of this procedure to see how the values in $T$ produce $\pi_T$.  The distinction between it and the full version will be discussed below.

Fix $\lambda \in \Lambda_n^+$ and $T \in \mathcal{T}_\lambda$.  The following inverse procedure will produce one permutation for each location in $\lambda$ between $(1,1)$ and $(\lambda_1, 1)$ inclusive in the reading order.  The locations are indicated as superscripts.  As the locations advance through $\lambda$, the permutations increase in the Bruhat order.  The first permutation is $\pi^{(1,1)} := (\pi_1^{(1,1)}, ... , \pi_n^{(1,1)})$ whose first $c_1$ entries are the values in $C_1$ (maintaining the increasing order), and whose final $n - c_1$ entries are the remaining values of $[n]$ (also in increasing order).  As $(j,i)$ advances from $(2,c_2)$ to $(\lambda_1,1)$, the procedure produces $\pi^{(j,i)}$ from $\pi^{(j,i+1)}$ based on the relationship between $C_{j-1}$ and $C_j$.  More specifically, the procedure uses $\pi^{(j,i+1)}$ to produce a permutation whose $i^{th}$ entry is $T(j,i)$, without changing any other entries from $\pi^{(j, i+1)}$ whose index is less than or equal to $c_j$.  From this we see that at an arbitrary location $(a,b) > (1,1)$, the permutation $\pi^{(a,b)}$ has $\pi^{(a,b)}_k = T(a-1,k)$ for $1 \leq k < b$, and $\pi^{(a,b)}_k = T(a,k)$ for $b \leq k \leq c_a$.  The procedure to produce $\pi^{(j,i)}$ from $\pi^{(j,i+1)}$ is as follows:

Since $T$ is semistandard, we have $T(j-1,i) \leq T(j, i)$.  If $T(j-1,i) = T(j,i)$, set $\pi^{(j,i)} = \pi^{(j,i+1)}$.  Otherwise, the semistandardness of $T$ and the previous paragraph imply that $T(j,i) = \pi^{(j,i+1)}_k$ for some $k > c_j$.  In this case, perform the following greedy algorithm:  Initialize the index $i_0 := i$, so we have $\pi^{(j,i+1)}_{i_0} = T(j-1,i)$.  Find the smallest index $i_1 > c_j$ such that $\pi^{(j,i+1)}_{i_0} < \pi^{(j,i+1)}_{i_1} \leq T(j,i)$.  If $\pi^{(j,i+1)}_{i_1} < T(j,i)$, find the smallest index $i_2$ such that $\pi^{(j,i+1)}_{i_1} < \pi^{(j,i+1)}_{i_2} \leq T(j,i)$.  Repeatedly obtain such indices until finding $i_m$ such that $\pi^{(j,i+1)}_{i_1} < ... < \pi^{(j,i+1)}_{i_m} = T(j,i)$. Then set $\pi^{(j,i)}_{i_x} := \pi^{(j,i+1)}_{i_{x-1}}$ for $1 \leq x \leq m$ and $\pi^{(j,i)}_{i_0} := \pi^{(j,i+1)}_{i_m}$.  For $1 \leq a \leq n$ such that $a \neq i_x$ for any $0 \leq x \leq m$, set $\pi^{(j,i)}_a := \pi^{(j,i+1)}_a$.  This completes the creation of $\pi^{(j,i)}$ and we say that the greedy algorithm has been executed at $(j,i)$.  Once the algorithm has been executed at $(\lambda_1,1)$, our final permutation $\pi^{(\lambda_1,1)} =: \pi_T$ has been produced.

For an example, refer to Figure 1 below.  Let $T$ be the first tableau in the figure.  The 10 (= 1 + 4 + 3 + 2) permutations displayed vertically are $\pi^{(1,1)}, \pi^{(2,4)}, \pi^{(2,3)},  ... , \pi^{(4,1)} = \pi_T$.  An asterisk on an entry in $\pi^{(j,i+1)}$ indicates that it is chosen during the production of $\pi^{(j,i)}$.

\begin{figure}[h!]
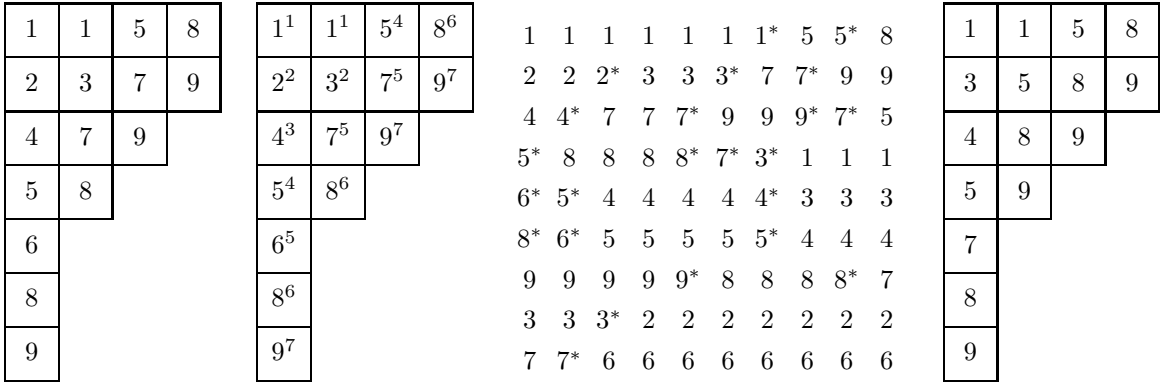

\ytableausetup{boxsize = 2em}
$$\begin{ytableau}
1 & 1 & 5 & 8  \\
2 & 3 & 7 & 9 \\
4 & 7 & 9 \\
5 & 8\\
6\\
8 \\
9 \\
\end{ytableau}\hspace{5mm}
\begin{ytableau}
1^1 & 1^1 & 5^4 & 8^6  \\
2^2 & 3^2 & 7^5 & 9^7 \\
4^3 & 7^5 & 9^7 \\
5^4 & 8^6\\
6^5\\
8^6 \\
9^7 \\
\end{ytableau}\hspace{5mm}
\ytableausetup{boxsize = 1.5 em}
\begin{ytableau}
\none[1]\\
\none[2]\\
\none[4]\\
\none[5^*]\\
\none[6^*]\\
\none[8^*]\\
\none[9]\\
\none[3]\\
\none[7]\\
\end{ytableau}
\begin{ytableau}
\none[1]\\
\none[2]\\
\none[4^*]\\
\none[8]\\
\none[5^*]\\
\none[6^*]\\
\none[9]\\
\none[3]\\
\none[7^*]\\
\end{ytableau}
\begin{ytableau}
\none[1]\\
\none[2^*]\\
\none[7]\\
\none[8]\\
\none[4]\\
\none[5]\\
\none[9]\\
\none[3^*]\\
\none[6]\\
\end{ytableau}
\begin{ytableau}
\none[1]\\
\none[3]\\
\none[7]\\
\none[8]\\
\none[4]\\
\none[5]\\
\none[9]\\
\none[2]\\
\none[6]\\
\end{ytableau}
\begin{ytableau}
\none[1]\\
\none[3]\\
\none[7^*]\\
\none[8^*]\\
\none[4]\\
\none[5]\\
\none[9^*]\\
\none[2]\\
\none[6]\\
\end{ytableau}
\begin{ytableau}
\none[1]\\
\none[3^*]\\
\none[9]\\
\none[7^*]\\
\none[4]\\
\none[5]\\
\none[8]\\
\none[2]\\
\none[6]\\
\end{ytableau}
\begin{ytableau}
\none[1^*]\\
\none[7]\\
\none[9]\\
\none[3^*]\\
\none[4^*]\\
\none[5^*]\\
\none[8]\\
\none[2]\\
\none[6]\\
\end{ytableau}
\begin{ytableau}
\none[5]\\
\none[7^*]\\
\none[9^*]\\
\none[1]\\
\none[3]\\
\none[4]\\
\none[8]\\
\none[2]\\
\none[6]\\
\end{ytableau}
\begin{ytableau}
\none[5^*]\\
\none[9]\\
\none[7^*]\\
\none[1]\\
\none[3]\\
\none[4]\\
\none[8^*]\\
\none[2]\\
\none[6]\\
\end{ytableau}
\begin{ytableau}
\none[8]\\
\none[9]\\
\none[5]\\
\none[1]\\
\none[3]\\
\none[4]\\
\none[7]\\
\none[2]\\
\none[6]\\
\end{ytableau}\hspace{5mm}
\ytableausetup{boxsize = 2em}
\begin{ytableau}
1 & 1 & 5 & 8  \\
3 & 5 & 8 & 9 \\
4 & 8 & 9 \\
5 & 9\\
7\\
8 \\
9 \\
\end{ytableau}$$
\caption{Example tableaux and permutations}
\end{figure}

\begin{lem}\label{amcliffincrease}Fix $1 \leq h \leq d+1$.  For $(j,i) \geq (\beta_{h},1)$ we have $\pi^{(j,i)}_{\zeta_{h-1} + 1} < ... < \pi^{(j,i)}_{\zeta_h}$.  \end{lem}

\begin{proof}For $1 \leq h \leq d$, it is easy to see that $\pi^{(\beta_h,1)}_x = T(\beta_h,x)$ for $1 \leq x \leq \zeta_h$.  So the column-strict condition on $T$ implies the result at $(\beta_h, 1)$.  Note that the $h=1$ case has been completely proven, since $\beta_1 = \lambda_1$.  When $h = d+1$, the definition of $\pi^{(1,1)}$ implies the result at $(1,1)$.  Now for any $1 < h \leq d+1$, assume that $\pi^{(j,i+1)}_{\zeta_{h-1} + 1} < ... < \pi^{(j,i+1)}_{\zeta_h}$ where $(j,i+1) \geq (\beta_h,1)$.  Let $x$ be the smallest index and $y$ the largest index satisfying $\zeta_{h-1}+1 \leq x \leq y \leq \zeta_h$ such that the greedy algorithm executed at $(j,i)$ chooses $\pi^{(j,i+1)}_x$ and $\pi^{(j,i+1)}_y$.  If no such indices exist the result follows trivially.  Otherwise, let $i \leq z \leq \zeta_{h-1}$ be the largest index less than $x$ such that $\pi^{(j,i+1)}_z$ is chosen by the greedy algorithm at $(j,i)$.  Clearly $\pi^{(j,i+1)}_{x-1} \leq \pi^{(j,i+1)}_z < \pi_x^{(j,i+1)}$, with equality in the former if and only if $z = x-1$.  If $x = y$, then $\pi^{(j,i+1)}_{x-1} \leq \pi^{(j,i+1)}_z < \pi^{(j,i+1)}_{x+1}$.  (If $x = \zeta_h$ the latter should be ignored.)  When $\pi^{(j,i)}$ is created $\pi^{(j,i)}_x = \pi^{(j,i+1)}_z$ and $\pi^{(j,i)}_a = \pi^{(j,i+1)}_a$ for all other $\zeta_{h-1}+1 \leq a \leq \zeta_h$, and so the result follows.

If $x < y$, the induction hypothesis implies that all of $\pi^{(j,i+1)}_x, \pi^{(j,i+1)}_{x+1}, ... , \pi^{(j,i+1)}_{y}$ are chosen by the greedy algorithm at $(j,i)$.  Then $\pi^{(j,i)}_x = \pi^{(j,i+1)}_z$, and $\pi^{(j,i)}_b = \pi^{(j,i+1)}_{b-1}$ for $x+1 \leq b \leq y$.  Further we have $\pi^{(j,i)}_a = \pi^{(j,i+1)}_a$ for $\zeta_{h-1}+1 \leq a < x$ and $y < a \leq \zeta_h$.  Since $\pi^{(j,i+1)}_{x-1} \leq \pi^{(j,i+1)}_z < \pi^{(j,i+1)}_x < ... < \pi^{(j,i+1)}_{y-1} < \pi^{(j,i+1)}_{y+1}$, the result follows for $\pi^{(j,i)}$.  \end{proof}

At $(j,i) = (\lambda_1, 1)$, the statement of the lemma for $1 \leq h \leq d+1$ says that:

\begin{cor}\label{pisnlambda} For any $T \in \mathcal{T}_\lambda$, we have $\pi_T \in S_n^\lambda$. \end{cor}

The remaining paragraphs in this section describe the connection between the above method and the full version of the inverse of the filling map.  First note that if $m > 1$ when the greedy algorithm is executed at $(j,i)$, then there are permutations between $\pi^{(j,i+1)}$ and $\pi^{(j,i)}$ in the Bruhat order.  The full version of the inverse of the filling map yields $m-1$ permutations between $\pi^{(j,i+1)}$ and $\pi^{(j,i)}$:  After obtaining $\pi^{(j,i+1)}$, produce $\pi^{(j,i+1;1)}$ by interchanging only $\pi^{(j,i+1)}_{i_0}$ and $\pi^{(j,i+1)}_{i_1}$ and leaving the remaining values unchanged.  The next permutation $\pi^{(j,i+1;2)}$ is produced by interchanging $\pi^{(j,i+1;1)}_{i_0}$ with $\pi^{(j,i+1;1)}_{i_2}$ (and leaving the rest unchanged).  This process repeats until $\pi^{(j,i+1;m-1)}_{i_0}$ is interchanged with $\pi^{(j,i+1;m-1)}_{i_m}$ to produce $\pi^{(j,i+1;m)} = \pi^{(j,i)}$.  The resulting chain $\pi^{(j,i+1)} < \pi^{(j,i+1;1)} < ... < \pi^{(j,i+1;m-1)} < \pi^{(j,i)}$ is saturated, as is the chain from $\pi^{(1,1)}$ to $\pi_T$ obtained in this fashion.

For an example, the full version of the inverse of the filling map produces 2 permutations between the second and third permutations ($\pi^{(2,4)}$ and $\pi^{(2,3)}$) displayed in Figure 1.  They are $\pi^{(2,4;1)} = (1,2,5,8,4,6,9,3,7)$ and $\pi^{(2,4;2)} = (1,2,6,8,4,5,9,3,7)$.

In addition, the full version of the inverse of the filling map begins with the identity permutation which we will denote $\pi^{(0,0)}$.  Let $C_0$ be a column of length $n$ whose value in row $i$ is $i$.  Then the process from the previous paragraph is executed at the locations $(1, c_1)$ through $(1,1)$ to obtain a saturated chain from $\pi^{(0,0)}$ to $\pi^{(1,1)}$.  Let $B_T$ denote the result of combining these chains to produce a saturated chain from $\pi^{(0,0)}$ to $\pi_T$.  Then $B_T$ is the saturated chain produced by the full version of the inverse of the filling map in the alcove model specialization in type A.

Lastly, for $c \in [n-1]$ define $\Gamma(c) := ( (c, c+1), (c, c+2), \dots , (c, n), (c-1, c+1), \dots , (c-1, n), \dots , (1, c+1), \dots , (1, n) )$.  Then define $\Gamma(\lambda)$ to be the concatenation of $\Gamma(c_1), \Gamma(c_2), \dots , \Gamma(c_{\lambda_1})$.  For $T \in \mathcal{T}_\lambda$, the set $\Gamma(\lambda)$ is the ordered list of all transpositions that may be applied to produce the permutations in $B_T$.  Let $x$ denote the number of elements in $\Gamma(\lambda)$, and consider the list to be indexed from 1 to $x$.  Each time a transposition in $\Gamma(\lambda)$ is applied during the creation of $B_T$, underline it.  Define $J_T \subseteq [x]$ to be the indices of the underlined transpositions.  Then $J_T$ is called an \emph{admissible subset}.  (The original \cite{Le1} definition for a subset of $[x]$ to be admissible is that its corresponding chain in the Bruhat order is saturated, which is a consequence of the construction above.)  Also note that given an admissible subset $J$, one can reverse this procedure to find its corresponding tableau $T_J$.  That is: given an admissible subset $J$, form the Bruhat chain corresponding to $J$ and let $T_J$ be the tableau whose $j^{th}$ column contains the first $c_j$ entries of $\pi^{(j,1)}$  for $1 \leq j \leq \lambda_1$.  See Section 3 of \cite{Le2} for full details.

\section{Equivalence of the scanning method and the type A alcove model}

Given $\lambda \in \Lambda_n^+$ and $T \in \mathcal{T}_\lambda$, fix a column $1 \leq l < \lambda_1$ and a location $(l,1) \leq (j,i) < (\lambda_1,1)$.  Consider the paths $P(T;l,k)$ originating in the $l^{th}$ column.  Let $U(T;l,j,i)$ denote the set of the most recent values in these paths $P(T;l,k)$ relative to $(j,i-1)$.  Let $U(T;l,\lambda_1,1)$ be the set of final values in these $P(T;l,k)$; these are the values $S(T;l,k)$ for $1 \leq k \leq c_l$.  Note that $T(j,i) \in U(T;l,j,i)$ for all $(j,i) > (l,1)$.  For the rightmost column only the final location is of interest.  Let $U(T;\lambda_1, \lambda_1, 1)$ be the set of values in the rightmost column of $T$ (and hence in $S(T)$ ).  For an example, let $T$ be the first tableau in Figure 1.  Let $l = 1$ and $(j,i) = (3,2)$.  Then $U(T;1,3,2) = \{ 1,3,4,5,7,8,9 \}$, based on the locations $(2, 1), (2,2), (1,3), (1,4), (3,2), (2,4), $ and $(3,3)$. For $T \in \mathcal{T}_\lambda$, let $\pi^{(1,1)} \leq ... \leq \pi^{(\lambda_1, 1)} = \pi_T$ be the (unsaturated) chain produced by the alcove model in Section 4.  For $(1,1) \leq (j,i) \leq (\lambda_1,1)$ and $1 \leq y \leq n$, define $\Pi_y^{(j,i)} := \{ \pi_1^{(j,i)}, ... , \pi_y^{(j,i)} \}$.

\begin{prop}\label{zetahequivalence}Fix $1 \leq h \leq d$.  For $(j,i) \geq (\beta_h,1)$, we have $\Pi_{\zeta_h}^{(j,i)} = U(T;\beta_h,j,i)$. \end{prop}

\begin{proof}It is easy to see that $\Pi_{\zeta_h}^{(\beta_h,1)} = \{ T(\beta_h, 1), ... , T(\beta_h, \zeta_h) \}$.  Further, the paths used to compute $U(T;\beta_h,\beta_h,1)$ contain only their initial locations from $C_{\beta_h}$.  So $U(T;\beta_h,\beta_h,1) = \{ T(\beta_h, 1), ... , T(\beta_h, \zeta_h) \}$, and the base case is established.  Again the $h=1$ case has been proven, so fix $1 < h \leq d$.  Let $(j,i+1) \geq (\beta_h,1)$ and assume that $\Pi_{\zeta_h}^{(j,i+1)} = U(T;\beta_h,j,i+1)$.  Let $\pi_i^{(j,i+1)} = \pi_{i_0}^{(j,i+1)} < \pi_{i_1}^{(j,i+1)} < ... < \pi_{i_m}^{(j,i+1)} = T(j,i)$ be the values chosen by the greedy algorithm at $(j,i)$.

If $\Pi_{\zeta_h}^{(j,i)} = \Pi_{\zeta_h}^{(j,i+1)}$, then $i_m \leq \zeta_h$ and $T(j,i) = \pi_a \in \Pi_{\zeta_h}^{(j,i+1)}$ for some $1 \leq a \leq \zeta_h$.  By the induction hypothesis there exists $T(l,k) \in U(T;\beta_h,j,i+1)$ such that $T(l,k) = \pi_a$.  By Lemma \ref{recentpath}, the path whose most recent value is $T(l,k)$ will append $(j,i)$.  So $T(l,k)$ is still the most recent value of this path, and the other paths are unchanged.  Thus $U(T;\beta_h,j,i) = U(T;\beta_h,j,i+1)$, and the induction hypothesis gives the result.

Assume $\Pi_{\zeta_h}^{(j,i)} \neq \Pi_{\zeta_h}^{(j,i+1)}$.  In this case $i_m > \zeta_h$, i.e. the value $T(j,i)$ is not in $\Pi_{\zeta_h}^{(j,i+1)}$.  The value of $\pi_i$ will change to $\pi_{i_m} = T(j,i)$, i.e. we have $T(j,i) \in \Pi_{\zeta_h}^{(j,i)}$.  Let $\pi_b$ denote the largest value in $\Pi_{\zeta_h}^{(j,i+1)}$ that is less than $T(j,i)$.  Recall that $\pi^{(j,i+1)}_k = T(j-1,k)$ for $1 \leq k \leq i$ and $\pi^{(j,i+1)}_k = T(j,k)$ for $i+1 \leq k \leq c_j$.  Thus the semistandardness of $T$ ensures that $b = i$ or $b > c_j$, and so $\pi_b$ will be chosen by the greedy algorithm.  By construction, the index $b$ is the largest index less than or equal to $\zeta_h$ that the greedy algorithm will choose.  In other words, the new index for the value $\pi_b$ in $\pi^{(j,i)}$ will be larger than $\zeta_h$.  Thus $\pi_b \notin \Pi_{\zeta_h}^{(j,i)}$.  In summary $\Pi_{\zeta_h}^{(j,i)} = (\hspace{1mm} \Pi_{\zeta_h}^{(j,i+1)} \backslash \{ \pi_b \} \hspace{1mm} ) \bigcup \hspace{1mm} \{ T(j,i) \}$.

By the induction hypothesis $\pi_b = T(l,k)$ for some $T(l,k) \in U(T;\beta_h,j,i+1)$.  By the construction of $\pi_b$, the value $T(l,k)$ is the largest value in $U(T;\beta_h,j,i+1)$ less than $T(j,i)$.  Again by Lemma \ref{recentpath}, the path whose most recent value is $T(l,k)$ will append $(j,i)$.  So $T(j,i)$ is now the most recent value of a path but $T(l,k)$ is not.  The remaining paths are unchanged.  In summary $U(T;\beta_h,j,i) = (\hspace{1mm} U(T;\beta_h,j,i+1) \backslash \{ T(l,k) \} \hspace{1mm}) \bigcup \hspace{1mm} \{ T(j,i) \}$.  Since $\pi_b = T(l,k)$, the induction hypothesis implies the result.  \end{proof}

For $1 \leq x \leq n$, write $\pi_{T_x}$ for the $x^{th}$ entry of $\pi_T$.  After the greedy algorithm has been executed at $(\lambda_1, 1)$ we have $\Pi_{\zeta_h}^{(\lambda_1, 1)} = \{ \pi_{T_1}, ... , \pi_{T_{\zeta_h}} \}$ for $1 \leq h \leq d$.  Also, the set $U(T;\beta_h,\lambda_1,1) = \{ S(T;\beta_h, 1) , ... , S(T;\beta_h, \zeta_h) \}$ for $1 \leq h \leq d$.

Recall that $\sigma_T$ is the permutation corresponding to $S(T)$.  Then for $\lambda \in \Lambda_n^+$ and $T \in \mathcal{T}_\lambda$, we have:

\begin{cor}\label{zetahequivalencecor}Fix $1 \leq h \leq d$.  Then $\{ \pi_{T_1} , ... , \pi_{T_{\zeta_h}} \} = \{ \sigma_{T_1}, ... , \sigma_{T_{\zeta_h}} \}$. \end{cor}

\begin{cor}\label{keyofpi} For any $T \in \mathcal{T}_\lambda$, we have $Y_\lambda(\pi_T) = S(T)$. \end{cor}

Corollary \ref{keyofpi} is the most direct way that we know of to show that $S(T)$ is indeed a key.  Since both $\pi_T$ and $\sigma_T$ are in $S_n^\lambda$, we see that the final permutation produced by the inverse of Lenart's filling map \cite{Le2} is the permutation that corresponds to the scanning tableau:

\begin{thm}\label{finalequivalence}Let $T \in \mathcal{T}_\lambda$ and let $\pi_T$ be the final permutation in the saturated Bruhat chain produced by the inverse of the filling map in the type A specialization of the alcove model.  Then $\pi_T = \sigma_T$. \end{thm}

Finally, since \cite{Wil} showed that $S(T) = R(T)$, we see that the key of the final permutation produced by the inverse of the filling map is the right key of $T$ introduced by Lascoux and Sch\"{u}tzenberger in \cite{LS}:

\begin{cor}\label{alcovetorightkey}Let $T \in \mathcal{T}_\lambda$.  Then $Y_\lambda(\pi_T) = R(T)$. \end{cor}

\section{Demazure Character Consequences}

To keep this paper purely combinatorial we use Theorem 1 of \cite{RS1} to define the Demazure character $d_{\lambda, w}(x)$, also known as the ``key polynomial''.  For $\lambda \in \Lambda_n^+$ and $w \in S_n^\lambda$, define the set of tableaux $\mathcal{D}_{\lambda, w} := \{ \hspace{1mm} T \in \mathcal{T}_\lambda \hspace{1mm} | \hspace{1mm} R(T) \leq Y_\lambda(w) \hspace{1mm} \} = \{ \hspace{1mm} T \in \mathcal{T}_\lambda \hspace{1mm} | \hspace{1mm} S(T) \leq Y_\lambda(w) \hspace{1mm} \}$.  For a set of indeterminates $x_1, ... , x_n$, define $x^T := x_1^{b_1} \dots x_n^{b_n}$ where $b_i$ is the number of $i$'s in $T$.  The \emph{Demazure character for $\lambda$ and $w$} is $d_{\lambda, w}(x) := \sum x^T$, where the sum is over $\mathcal{D}_{\lambda, w}$.

For the connection to representation theory see, for example, the appendix of \cite{PW}.  In particular, all of the Demazure characters for $sl_n(\mathbb{C})$ can be produced in this fashion (with their exponents shifted to be integers).

Theorem \ref{demchar} below does follow immediately from Corollary \ref{alcovetorightkey}, but a stronger connection can be established:  In \cite{RS2} Reiner and Shimozono present a tableau description of a condition from \cite{LMS} to determine whether or not a given tableau is in $\mathcal{D}_{\lambda, w}$:  Fix $T \in \mathcal{T}_\lambda$ and denote its columns $C_1, ... , C_{\lambda_1}$.  Let ${(1^x)}$ be the partition consisting of $x$ ones, i.e. whose Young diagram is a single column of length $x$.  A \emph{defining chain for $T$} is a sequence of weakly increasing elements in the Bruhat order $w^1 \leq ... \leq w^{\lambda_1}$ such that $Y_{(1^{c_j})}(w^j) = C_j$ for $1 \leq j \leq \lambda_1$.  Then $T \in \mathcal{D}_{\lambda, w}$ if and only if there exists a defining chain $\{ w^j \}$ for $T$ with $w^{\lambda_1} \leq w$ in the Bruhat order.  Lemma 7 of \cite{RS2} is attributed to Deodhar and states that every $T$ has a minimal defining chain.  That is, there exists a defining chain $w^1 \leq ... \leq w^{\lambda_1}$ for $T$ such that if $v^1 \leq ... \leq v^{\lambda_1}$ is a defining chain for $T$ then $w^j \leq v^j$ for all $1 \leq j \leq \lambda_1$.  Thus $T \in \mathcal{D}_{\lambda, w}$ if and only if its minimal defining chain $\{ w^j \}$ has $w^{\lambda_1} \leq w$.

They then define the \emph{canonical lift} $w(T)$ of $T$ to be the shortest permutation in the Bruhat order such that $Y_{(1^{c_j})}(w(T))$ equals the $j^{th}$ column of $R(T)$ for $1 \leq j \leq \lambda_1$.  Let $\lambda^j$ and $T^j$ denote the results of removing $C_{j+1}, ...  , C_{\lambda_1}$ from $\lambda$ and $T$ respectively.  Lemma 8 of \cite{RS2} states that the minimal defining chain $w^1 \leq ... \leq w^{\lambda_1}$ for $T$ has $w^j$ equal to the canonical lift of $T^j$ for $1 \leq j \leq \lambda_1$.  Equivalently, in the minimal defining chain $w^j$ is the shortest permutation such that $Y_{\lambda^j}(w^j) = S(T^j)$ for $1 \leq j \leq \lambda_1$.

\begin{prop}\label{minimaldefchain}Fix $T \in \mathcal{T}_\lambda$.  Let $\{ w^j \}$ be its minimal defining chain, and let $\pi^{(1,1)} \leq ... \leq \pi^{(\lambda_1, 1)} = \pi_T$ be its Bruhat chain produced by the alcove model.  Then $w^j = \pi^{(j,1)}$ for $1 \leq j \leq \lambda_1$.\end{prop}

\begin{proof}For $1 \leq j \leq \lambda_1$, the location $(j,1)$ is the last location in $\lambda^j$.  Since the production of $\pi_T$ applies the greedy algorithm as the locations advance from $(1,1)$ to $(\lambda_1,1)$, the construction of $\pi^{(j,1)}$ is independent of the columns to the right of $C_j$.  In other words $\pi^{(j,1)} = \pi_{T^j}$, the final permutation when the greedy algorithm is applied to $T^j$.  Thus Corollary \ref{keyofpi} gives $Y_{\lambda^j}(\pi^{(j,1)}) = S(T^j)$.  By Corollary \ref{pisnlambda} we have $\pi^{(j,1)} \in S_n^{\lambda^j}$, so $\pi^{(j,1)}$ is the shortest permutation satisfying the previous condition. \end{proof}

As an example, when $T$ is the first tableau in Figure 1, the first, fifth, eighth, and tenth permutations in Figure 1 make up the minimal defining chain for $T$.  Applying the proposition for $j = \lambda_1$ we see that:

\begin{thm}\label{demchar}Let $T \in \mathcal{T}_\lambda$ and $w \in S_n^\lambda$.  Then $T \in \mathcal{D}_{\lambda, w}$ if and only if $\pi_T \leq w$. \end{thm}

Lastly, Theorem 3.6 (3) of \cite{Le1} presents a Demazure character formula for $\lambda$ and $w$ obtained by summing the weights of all admissible subsets whose corresponding saturated chain has its final permutation less than or equal to $w$.  Let $Ad(\lambda, w)$ denote the set of admissible subsets from this theorem that contribute to the Demazure character for $\lambda$ and $w$.  It is known that the process of constructing the tableau $T_J$ from an admissible subset $J$ described in Section 4 is a bijection from $Ad(\lambda, w)$ to the set of corresponding tableaux.  Define $\mathcal{A}_{\lambda, w} := \{ T_J \in \mathcal{T}_\lambda \hspace{1mm} | \hspace{1mm} J \in Ad(\lambda, w) \}$.  The following result is clear once the above constructions are understood:

\begin{cor}Fix $\lambda \in \Lambda_n^+$ and $w \in S_n^\lambda$.  Then $\mathcal{D}_{\lambda, w} = \mathcal{A}_{\lambda, w}$.\end{cor}

\begin{proof}The process of finding the minimal defining chain for $T$, expanding it to the saturated chain $B_T$ via the inverse of the filling map, and subsequently finding its admissible subset $J_T$ is the desired bijection between the generating sets from $\mathcal{D}_{\lambda, w}$ to $Ad(\lambda, w)$.\end{proof}

\vspace{1pc}\noindent \textbf{Acknowledgments.}  The author would like to thank Cristian Lenart, whose helpful discussions inspired the results of this paper.  I would also like to thank Bob Proctor for always helping me see the forest from the trees; in particular his comments motivated some of the discussions in Section 6, and he suggested several expositional improvements.  Thanks also to David Lax for his helpful suggestions on the first draft of the paper.

\bibliographystyle{amsalpha}

\end{document}